\documentclass[a4paper,12pt]{amsart}

  \usepackage{amssymb,amsthm}
\usepackage{graphicx,color}        
 
\usepackage{url} 

\newtheorem{theorem}{Theorem}[section]
\newtheorem{corollary}[theorem]{Corollary}
\newtheorem{proposition}[theorem]{Proposition}

\theoremstyle{definition}
\newtheorem{definition}[theorem]{Definition}
\newtheorem{example}[theorem]{Example}
\newtheorem{remark}[theorem]{Remark}

\def\R{\mathbb{R}}
  
\def\n{\textbf{n}}

\def\v{\textbf{v}}
\def\a{\textbf{a}}

\begin{document}

\title{The Plateau-Rayleigh instability of translating $\lambda$-solitons}
\author{Antonio Bueno, Rafael L\'opez, Irene Ortiz}
\address{ Departamento de Ciencias\\  Centro Universitario de la Defensa de San Javier. 30729 Santiago de la Ribera, Spain}
\email{antonio.bueno@cud.upct.es}
\address{ Departamento de Geometr\'{\i}a y Topolog\'{\i}a\\  Universidad de Granada. 18071 Granada, Spain}
\email{rcamino@ugr.es}

\address{ Departamento de Ciencias\\  Centro Universitario de la Defensa de San Javier. 30729 Santiago de la Ribera, Spain}
\email{irene.ortiz@cud.upct.es}

\begin{abstract} 
Given a unit vector $\v\in\R^3$ and $\lambda\in\R$, a translating $\lambda$-soliton is a surface in $\mathbb{R}^3$ whose mean curvature $H$ satisfies $H=\langle N,\v\rangle+\lambda,\ |\v|=1$, where $N$ is the Gauss map of the surface. In this paper, we extend the phenomenon of instability of Plateau-Rayleigh for translating $\lambda$-solitons of cylindrical type, proving that long pieces of these surfaces are unstable. We will provide explicit bounds on the length of these surfaces. It will be also proved that if a translating $\lambda$-soliton  is a graph, then it is a minimizer of the weighted area in a suitable class of surfaces with the same boundary and the same weighted volume.  
\end{abstract}

\maketitle
\section{Introduction  and formulation of the problem}
In 1873, Plateau observed that a stream of water dropping vertically was divided into smaller packets with the same volume but less surface area \cite{pl}. Plateau found experimentally that this happened when the length of the stream was greater than about $3'13$ times its diameter. Later Rayleigh proved theoretically that a stream of radius $r$ breaks into drops if its length is greater than $2\pi r$ \cite{ra}.This behavior is known as the Plateau-Rayleigh instability and it is part of a greater branch of fluid dynamics concerned with fluid thread breakup \cite[Ch. 5]{de}.

This physical phenomenon can be explained  within the theory of surfaces with constant mean curvature (cmc surfaces in short). Constant mean curvature surfaces are critical points of   the area functional for any volume preserving variation of the surface.   Let     $C_r$ be a circular cylinder of radius $r>0$, whose mean curvature is $H=1/r$ and consider a piece of $C_r$ of length $L$. The theory of stability of cmc surfaces asserts that if $L>2\pi r$, then the cylinder $C_r$ is unstable \cite{mc}: see also Section \ref{sec5} below.   From the mathematical viewpoint, the Plateau-Rayleigh criterion has appeared in contexts others as cmc surfaces: \cite{ag,ba,bs,gr,lo0,lo3,mc0}.

The purpose of this paper is to extend the Plateau-Rayleigh phenomenon  to  translating $\lambda$-solitons. Translating $\lambda$-solitons  are critical points of a weighted area functional and a notion of stability arises in a natural way. We precise the necessary background. Let $\R^3$ be the Euclidean $3$-space with its usual metric $\langle\cdot,\cdot\rangle$. Let $dA$ and $dV$ be the usual area and volume element, respectively. Given $\phi\in C^\infty(\R^3)$, the triple $(\R^3,\langle\cdot,\cdot\rangle,e^\phi dV)$ is known as the weighted Euclidean space with density $\phi$, whose weighted area and volume elements are given by 
$$dA_\phi=e^\phi\,dA,\hspace{.5cm} dV_\phi=e^\phi\,dV.$$
As in the theory of cmc surfaces,  it is natural to study the critical points of the weighted area $A_\phi$.  Let $\Sigma$ be a surface  in $\R^3$ and consider $\{\Sigma_t:|t|<\epsilon\}$  a variation of $\Sigma$, $\Sigma_0=\Sigma$, whose variational field $\xi$ has compact support. Let $u=\langle \xi,N\rangle\in C_0^\infty(\Sigma)$ be the normal component of $\xi$, where $N$ stands for the Gauss map of $\Sigma$. We also assume that the boundary $\partial\Sigma_t$ of $\Sigma_t$ coincides with   $\partial\Sigma$ for any $t\in(-\epsilon,\epsilon)$.  Let $A_\phi(t)$ and $V_\phi(t)$ be the weighted area and volume functionals associated to the variation.   The formulas of the first variations of $A_\phi$ and $V_\phi$  are
\begin{equation}\label{firstvariations}
 A'_\phi(0)=-\int_\Sigma (H-\langle N,\overline{\nabla}\phi\rangle)u\,dA_\phi,\quad  V'_\phi(0)=-\int_\Sigma u\,dV_\phi,
\end{equation}
where $H$ is the mean curvature of $\Sigma$ defined as the sum of the principal curvatures. Here $\overline{\nabla}$ is the gradient in $(\R^3,\langle\cdot,\cdot\rangle)$. See \cite{bay} for a detailed derivation of these formulas. The variation is said to be volume preserving  if $V_\phi(t)$ is constant for any $t$, or equivalently, if
$$
\int_\Sigma u\,dA_\phi=0.
$$
From the expression of $A'(0)$, the  weighted mean curvature  of $\Sigma$ is defined as  
$$
H_\phi=H-\langle N,\overline{\nabla}\phi\rangle.
$$
It is immediate from  \eqref{firstvariations} that $A_\phi'(0)=0$ for any volume preserving variation if and only if  $H_\phi=\lambda$ is constant on $\Sigma$, $\lambda\in\R$. If we drop the condition that the variations preserve the volume, then $\Sigma$ is a critical point of $A_\phi$ if and only if  $H_\phi=0$.
 
In this paper, we focus on the particular density $\phi(x)=\langle x,\v\rangle$, $x\in\R^3$, where $\v\in\R^3,|\v|=1$. Then, $\overline{\nabla}\phi=\v$ and we give the following definition for the critical points of $A_\phi$.

\begin{definition} 
Let $\v\in\R^3$, $|\v|=1$, called the density vector. 
An orientable surface $\Sigma$ in $\R^3$ is called a translating $\lambda$-soliton  with respect to $\v$ if   $H_\phi=\lambda$ is constant on $\Sigma$, or equivalently, if 
\begin{equation}\label{eq2}
H=\langle N,\v\rangle+\lambda.
\end{equation}
\end{definition}

If $\lambda=0$ we simply refer to these surfaces as translating solitons. This particular case $\lambda=0$ is of special interest because, in a different context, translating solitons appear  as eternal solutions of the mean curvature flow in the sense that they evolve by pure translations during the flow. From this viewpoint, these surfaces have been widely studied because they are models of the singularities of type II in the theory of the mean curvature flow \cite{hui1,hui2}.

The general case $\lambda\not=0$ has received less attention, despite this class of surfaces is somehow equivalent to cmc surfaces, at least from the variational viewpoint (as well as translating solitons are to minimal surfaces). For example, in \cite{lo1}, the second author classified the translating $\lambda$-solitons invariant by a one-parameter group of translations and a one-parameter group of rotations.  It was  proved the existence of rotational translating $\lambda$-solitons that intersect orthogonally the rotation axis (see \cite{bo} for similar results for hypersurfaces). In \cite{lo2} it was studied compact  translating $\lambda$-solitons with boundary. Some results of rigidity have been also obtained  assuming that the second fundamental form has constant square norm \cite{li} or that the surface is of Weingarten type \cite{fu}. 

A translating $\lambda$-soliton  $\Sigma$ is said to be {\it stable} (resp. {\it strongly stable}) if $A_\phi''(0)\geq 0$ for any volume preserving variation  (resp. variation) of $\Sigma$.  The formula of the second variation of $A_\phi$ is 
 \begin{equation}\label{a2}
A''_\phi(0)=-\int_\Sigma u(\Delta u+\langle\nabla u,\v\rangle+|A|^2 u)\,dA_\phi,
\end{equation}
where $u=\langle\xi,N\rangle\in C_0^\infty(\Sigma)$. Here $\Delta$ and $\nabla$ are the Laplace-Beltrami and gradient operators, both computed in $\Sigma$, and $A$ is the second fundamental form of $\Sigma$.

 Although the study of stability of translating solitons  ($\lambda=0$) has been considered in the literature (\cite{css,gu,im,ku,mm,sh,xi}),   there are no results on the stability of   translating $\lambda$-solitons for $\lambda\neq0$. In this paper we investigate the Plateau-Rayleigh instability criterion for translating $\lambda$-solitons.  
 
 Let $\Sigma$ be a surface in $\R^3$ of  cylindrical type. Then, $\Sigma$ is parametrized by 
\begin{equation}\label{para}
\Psi(s,t)=\alpha(s)+t\, \a,\quad \a\in\R^3,\ |\a|=1,
\end{equation}
where $\alpha:I\rightarrow\R^3$ is an arc-length parametrized curve contained in a vectorial plane $\Pi$ orthogonal to the  ruling direction  $\a$. The curve $\alpha$ is called the base curve of $\Sigma$. Let us fix an interval $[a,b]$ in the domain of  $\alpha$. Since $\Sigma$ is invariant along the $\a$-direction, we take a piece of $\Sigma$ of length $L>0$ in the $t$-parameter, that is, $\Sigma(a,b;L)=\Psi([a,b]\times[0,L])$. Suppose now that $\Sigma$ is a translating $\lambda$-soliton. Motivated by the instability phenomenon of circular cylinders, it is natural  to expect that for small values of $L$, the surface $\Sigma(a,b;L)$ is stable but if $L$ is sufficiently large, then the surface is unstable. In this paper, we shall answer the following question: 

\begin{quote}
\emph{Given  $[a,b]\subset I$, does there exist $L_0>0$ such that the compact piece $\Sigma(a,b;L)$ is unstable for any $L>L_0$?}
\end{quote}

Comparing with the cmc case, the main difficulty is that the expression $A''_\phi(0)$ in \eqref{a2} is more complicated because there appears a first order term, $\langle\nabla u,\v\rangle$, and also because  of the weighted area element $dA_\phi$. However, a point that will help our work is that there are explicit parametrizations of the base curves of  translating $\lambda$-solitons of cylindrical type.

The stability operator of translating $\lambda$-solitons will be studied in Section  \ref{sec3}.  As a first result, we prove that graphical translating $\lambda$-solitons are strongly stable (Theorem \ref{t1}), extending the case $\lambda=0$ proved in \cite{sh}. It will be also proved that graphical translating $\lambda$-solitons are minimizers of the weighted area in a suitable class of   surfaces with the same boundary and weighted volume (Theorem \ref{t2}).   In Section \ref{sec4}  we will address the Plateau-Rayleigh instability criterion initially posed by distinguishing between the cases $\lambda>1$, $\lambda=1$ and $\lambda<1$. This distinction raises in a natural way by explicit integration of equation \eqref{eq2} for a cylindrical translating $\lambda$-soliton. Finally  in Section \ref{sec5} we will study the stability of a circular cylinder which is both a translating $\lambda$-soliton and a cmc surface. We will compare the length of cylinders in the Plateau-Rayleigh instability criterion for both notions of stability.

\section{ Translating $\lambda$-solitons of cylindrical type}\label{sec2}

In this section, we present the class of translating $\lambda$-solitons of cylindrical type.   Let $(x_1,x_2,x_3)$ stand for the canonical coordinates of $\R^3$. Let $\Sigma$ be a  cylindrical surface parametrized as in \eqref{para}. If $\Sigma$ is a translating $\lambda$-soliton, equation \eqref{eq2} can be expressed in terms of the base curve $\alpha$. The principal curvatures   of $\Sigma$ are $\kappa_\alpha$ and $0$, where $\kappa_\alpha$ is the curvature of $\alpha$. The Gauss map $N$ of $\Sigma$ is the unit normal vector $\n_\alpha$ of $\alpha$, that is, $N(\Psi(s,t))=\n_\alpha(s)$. Then, the translating $\lambda$-soliton equation  \eqref{eq2} for $\Sigma$ can be expressed as 
\begin{equation}\label{1di}
\kappa_\alpha(s)=\langle\n_\alpha(s),\v\rangle+\lambda.
\end{equation}

 Two examples are planes and circular cylinders.
 
\begin{example}\label{r1} 
Given a density vector $\v$, any plane is a translating $\lambda$-soliton for $\lambda=-\langle N,\v\rangle$.
\end{example}

\begin{example}\label{r2} A circular cylinder of radius $r>0$ and ruling direction $\a$ is a translating $\lambda$-soliton with density vector $\v=\a$ and $|\lambda|=1/r$. Notice that in \eqref{eq2}, $\langle N,\v\rangle=0$ and $|H|=1/r$. As a consequence, a circular cylinder is a cmc surface and also it is a translating $\lambda$-soliton. In fact, if a cmc surface $\Sigma$ is a translating $\lambda$-soliton, then  $\langle N,\v\rangle$ is a constant function. Then, $\Sigma$ is a constant angle surface and the only constant angle surfaces with non-zero constant mean curvature are circular cylinders \cite{mn}.
\end{example}

 Solutions of \eqref{1di} are explicitly known  \cite{hn,lo1}. Let us point out that  there is no a priori relation between the density vector $\v$ and the ruling direction $\a$ of $\Sigma$.  This observation also occurs for translating solitons where if, for example, $\v=(0,0,1)$, there are examples of cylindrical type  whose ruling direction $\a$ goes from orthogonal to $\v$ (grim reaper), when $\a$ is not horizontal (tilted grim reapers) and when $\a$ is vertical (plane). 
 
In this paper we will assume that the rulings are orthogonal to the density vector $\v$, since otherwise the statements of them are a bit cumbersome. After a change of Euclidean coordinates and a rotation of $\R^3$, we can assume that $\v=e_3=(0,0,1)$ and $\a=e_2=(0,1,0)$, where $\{e_1,e_2,e_3\}$ is the canonical basis of $\R^3$. Also, after a change of the orientation of $\Sigma$, we can assume $\lambda\geq 0$. Thus, the parametrization \eqref{para} of $\Sigma$ is now
\begin{equation}\label{para2}
\Psi(s,t)=(\alpha_1(s),t,\alpha_3(s)),
\end{equation}
and the translating $\lambda$-soliton  equation \eqref{1di}  writes as   
\begin{equation}\label{eq1}
\kappa_\alpha(s)=  \alpha_1'(s)+\lambda.
\end{equation}
The Gauss map of $\Sigma$ is    $N=\n_\alpha=(-\alpha_3',0,\alpha_1')$.  The solutions of \eqref{eq1} can be obtained explicitly. Since $\alpha$ is parametrized by arc-length, there is a function $\theta(s)$ such that   $\alpha'(s)=(\cos\theta(s),0,\sin\theta(s))$. Then, \eqref{eq1} writes as 
$\theta'(s)= \cos\theta(s)+\lambda$, hence 
$$\int\frac{d\theta}{\cos\theta+\lambda}=s+c,\quad c\in\R.$$
This integral can be solved by quadratures obtaining $\theta(s)$ and, consequently, the curve $\alpha(s)$.  In the following result, we explicit these solutions: see \cite[Th. 2.6]{lo1}.

\begin{proposition} \label{pr1}
The base curve of a translating $\lambda$-soliton parametrized by \eqref{para2}   has the following parametrization: 
\begin{enumerate}
\item Case $\lambda>1$. Then,
\begin{equation}\label{fp1}
\alpha(s)=\left(- \lambda s+2\arctan\left(\sqrt{\tfrac{\lambda+1}{\lambda-1}}\tan (\tfrac{s}{2}\sqrt{\lambda^2-1})\right), 0,
 \log (\lambda- \cos (s\sqrt{\lambda^2-1} ))\right).
 \end{equation}
The domain of $\alpha$ is $\R$. The curve $\alpha$ is  not embedded neither closed. Moreover, $\alpha$ is periodic along the $x_1$-axis with period $T=2\pi/\sqrt{\lambda^2-1}$ in the parameter domain.
\item Case $\lambda=1$.  Then,
\begin{equation}\label{fp2}
\alpha(s)=\left(-s+2\arctan(s),0, \log(1+s^2)\right).
\end{equation}
The domain of $\alpha$ is $\R$. The curve $\alpha$ is symmetric about the $x_3$-axis and it  has a unique self-intersection point. The   normal vector $\emph{\n}_\alpha$ converges to $ (0,0,-1)$ as $|s|\to\infty$. 
\item Case $\lambda<1$. There are two sub-types of curves. A type corresponds with curves that are graphs on the $x_1$-axis. A second type of curves are   parametrized by 
\begin{equation}\label{fp3}
\alpha(s)=\left(- \lambda s+2\arctan\left(\sqrt{\tfrac{1+\lambda}{1-\lambda}}\tanh(\tfrac{s}{2}\sqrt{1-\lambda^2})\right),0, \log (-\lambda+  \cosh (s\sqrt{1-\lambda^2} )\right).
\end{equation}
The domain of $\alpha$ is $\R$. The curve $\alpha$ is symmetric about the $x_3$-axis  and it  has a unique self-intersection point.  The unit normal $\emph{\n}_\alpha$ converges to $ (\pm\sqrt{1-\lambda^2},0,-\lambda)$.
\end{enumerate}
\end{proposition}

\section{Stability on translating $\lambda$-solitons}\label{sec3}

Let $\Sigma$ be a translating $\lambda$-soliton.  From the expression \eqref{a2} for $A_\phi''(0)$, we define a second order operator $L_\phi$ acting on $C_0^\infty(\Sigma)$ by
\begin{equation}\label{ll}
L_\phi [u]=\Delta u+\langle\nabla u,\v\rangle+|A|^2 u.
\end{equation}
In consequence, a quadratic form $Q_\phi$  can be defined by
\begin{equation}\label{sl}
Q_\phi(u)=-\int_\Sigma u  L_\phi[u]\, dA_\phi.
\end{equation}
Notice that $L_\phi$ is not self-adjoint with respect to the standard $L^2$ inner product, but it is with respect to the weighted inner product $\int_\Sigma uv\, dA_\phi$. The following result follows immediately by taking the vector field $u\nabla u$ and applying in \eqref{sl} the divergence theorem for manifolds with density.

\begin{proposition}
A translating $\lambda$-soliton $\Sigma$  is strongly stable if and only if for any   $u\in C_0^\infty(\Sigma)$, we have
\begin{equation}\label{defQ}
Q_\phi(u)=-\int_\Sigma uL_\phi [u]\,dA_\phi=\int_\Sigma(|\nabla u|^2-|A|^2u^2)\,dA_\phi\geq0.
\end{equation}
If, in addition, \eqref{defQ} holds for any $u\in C_0^\infty(\Sigma)$ with $\int_\Sigma u\, dA_\phi=0$, then $\Sigma$ is  stable. 
\end{proposition}

\begin{example} Planes are  strongly stable because $A=0$, hence \eqref{defQ} is non-negative.
\end{example}

The first result of stability of translating $\lambda$-solitons concerns the strong stability of graphs. For $\lambda=0$, the following result was proved in  \cite{sh} but the same arguments hold for arbitrary $\lambda$.

  \begin{theorem}\label{t1}
If $\Sigma$ is a graphical translating $\lambda$-soliton,  then $\Sigma$ is  strongly stable.
\end{theorem}
 
\begin{proof} It is well known that the Gauss map  $N$ of a surface $\Sigma$ of $\R^3$ satisfies 
\begin{equation}\label{nn}
\Delta N+|A|^2 N=-\nabla H.
\end{equation}
If $\Sigma$ is a translating $\lambda$-soliton, we have $\nabla H=\langle\nabla N,\v\rangle$ by  \eqref{eq2}. By hypothesis, $\Sigma$ is a graph over some plane $\Pi$. Let $\textbf{w}$  be a unit vector orthogonal to the plane $\Pi$ and define the function  $v=\langle N, \textbf{w}\rangle$, which does not vanish on $\Sigma$ because $\Sigma$ is a graph. From \eqref{nn}, the function $v$  satisfies 
\begin{equation*}
\begin{split}
\Delta v+|A|^2 v&=-\langle \nabla \langle N,\v\rangle,\textbf{w}\rangle= -A(\v^\top,\textbf{w}^\top)=-\langle  \nabla \langle N,\textbf{w}\rangle,\v\rangle\\
&=-\langle \nabla v,\v\rangle.
\end{split}
\end{equation*}
Therefore, $v$ is a non-vanishing function on $\Sigma$ such that $L_\phi [v]=0$. We now follow standard techniques using the log of the function $ v$, according to arguments that appear in \cite{fc}. We give briefly the details. Let   $u\in C_0^\infty(\Sigma)$. The function  $h=\log v$ satisfies
$$
-\Delta h=\langle\nabla h,\textbf{w}\rangle+|A|^2+|\nabla h|^2.
$$
Multiplying by $u^2$ and integrating in the weighted space, 
\begin{equation}\label{ab}
\int_\Sigma u^2\langle\nabla h,\textbf{w}\rangle\,  dA_\phi+\int_\Sigma u^2|A|^2\, dA_\phi+\int_\Sigma u^2|\nabla h|^2\, dA_\phi=-\int_\Sigma u^2\Delta h\, dA_\phi.
\end{equation}
Applying the divergence theorem for weighted manifolds to the function $u^2 \nabla h$, and using \eqref{ab}, we have
$$
0=\int_\Sigma 2u \langle\nabla h,\nabla u\rangle dA_\phi-\int_\Sigma u^2(|\nabla h|^2+|A|^2)\, dA_\phi.
$$
Since $2u \langle\nabla h,\nabla u\rangle\leq|\nabla u|^2+u^2|\nabla h|^2$, the above identity implies 
$$\int_\Sigma u^2(|\nabla h|^2+|A|^2)\, dA_\phi\leq \int_\Sigma\left(|\nabla u|^2+u^2|\nabla h|^2\right)\, dA_\phi,$$
which leads directly to $
Q_\phi(u)\geq0$. The arbitrariness of $u\in C_0^\infty(\Sigma)$ allows   to conclude the result.\end{proof}

 The statement of Theorem \ref{t1} can be improved by exhibiting that graphical translating $\lambda$-soliton are not only strongly stable, but actually minimizers within the class of all surfaces with the same boundary and the same weighted volume.  The argument in the following result uses techniques of  calibrations in the context of manifolds with density.

\begin{theorem}\label{t2}  Let $\Sigma$ be a compact translating $\lambda$-soliton with density $\emph{\v}$ which is a graph on a domain $U$ of a plane $\Pi=\emph{\v}^\top$.   Then, $\Sigma$ is a minimizer of the weighted area $A_\phi$ in the class of all surfaces included in the cylinder  $U\times\R\emph{\v}$ with the same boundary and  homology class and the same weighted volume $V_\phi$    than $\Sigma$.
\end{theorem}

\begin{proof} 
After a change of coordinates, we can suppose $\v=(0,0,1)$ and that $\Pi$ is the plane of equation $x_3=0$.  
Suppose that $\Sigma$ is a graph of a function $u\in C^\infty(U)$.  The expression of the  Gauss map $N$ of $\Sigma$ is    given by
$$N(x_1,x_2,u(x_1,x_2))=\frac{1}{\sqrt{1+| Du|^2}}\left(-Du,1\right).$$
On the domain $U\times\R\subset\R^{3}$, define the vector field 
$$X(x_1,x_2,x_3)=e^{x_3}N(x_1,x_2,u(x_1,x_2))=\frac{e^{x_3}}{\sqrt{1+| Du|^2}}\left(-Du,1\right).$$
Let us observe that at each point $(x_1,x_2,x_3)\in U$, the vector  field $X(x_1,x_2,x_3)$ is  the vertical translation of the Gauss map   at $(x_1,x_2,u(x_1,x_2))\in \Sigma$ multiplied by $e^{x_3}$. Using \eqref{eq2}, we have
\begin{equation}\label{14}
\begin{split}
\mbox{Div}_{\R^{3}}(X)&=e^{x_3}\mbox{Div}_{\R^2}\left(\frac{-Du}{\sqrt{1+| Du|^2}}\right)+e^{x_3}\langle N,\v\rangle\\
&=e^{x_3}\left(-H+\langle N,\v\rangle\right)  =-\lambda e^{x_3}.
\end{split}
\end{equation}
  
Let $\widetilde{\Sigma}$ be a surface included in the cylinder $U\times\R$ with the same boundary and the same homology class than $\Sigma$.   We also assume that $\widetilde{\Sigma}$ encloses the same weighted volume $V_\phi$ than $\Sigma$. Since $\Sigma$ and $\widetilde{\Sigma}$ have the same homology class, $\Sigma\cup\widetilde{\Sigma}$ define a $3$-chain $\Omega$ in $\R^3$, which is included in $U\times\R$. Let $\tilde{N}$ be  the compatible orientation of $\tilde{\Sigma}$.  Using \eqref{14}, the divergence theorem yields
\begin{eqnarray*}
-\int_{\Omega} \lambda e^{x_3} \, dV &=&\int_{\Omega}\mbox{Div}_{\R^{3}}(X)\, dV=\int_\Sigma\langle X,N\rangle\, d\Sigma+\int_{\widetilde{\Sigma}}\langle X,\widetilde{N}\rangle\, d\widetilde{A}\nonumber\\
&=&\int_\Sigma e^{x_3}\, d\Sigma+\int_{\widetilde{\Sigma}}e^{x_3}  \langle N,\widetilde{N}\rangle\, d\widetilde{A}\geq  \int_\Sigma e^{x_3}\, d\Sigma-\int_{\widetilde{\Sigma}}e^{x_3} d\widetilde{A}\nonumber\\
& =&A_\phi(\Sigma)-A_\phi(\widetilde{\Sigma}),
\end{eqnarray*}
because $\langle N,\widetilde{N}\rangle\leq 1$.  On the other hand, the surfaces $\Sigma$ and $\widetilde{\Sigma}$ enclose the same weighted volume. Thus,
$$\int_\Omega\lambda e^{x_3}\, dV=\lambda\int_\Omega e^{x_3}\, dV=\lambda\int_\Omega dV_\phi=0.$$
This finishes the proof. 
\end{proof}
 
 In the particular case $\lambda=0$, the assumption on the constancy of the weighted volume can be dropped.  
  
\begin{corollary}   Let $\Sigma$ be a compact translating  soliton with density $\emph{\v}$. If $\Sigma$ is a graph on a  plane $\Pi=\emph{\v}^\top$, then  $\Sigma$ is a minimizer of the weighted area $A_\phi$ in the class of of all surfaces included in the cylinder  $U\times\R\emph{\v}$   with the same boundary  and the same homology class than $\Sigma$.
\end{corollary}

\section{The Plateau-Rayleigh instability criterion} \label{sec4}

In this section, we address the Rayleigh-Plateau instability for translating $\lambda$-solitons.  Let $\Sigma$ be a cylindrical translating $\lambda$-soliton parametrized by  \eqref{para2}. Recall that the density vector is $\v=e_3$.  The desired instability result will be achieved once we determine a compact piece $\Sigma^*\subset\Sigma$ and  a function $u\in C_0^\infty(\Sigma^*)$ such that
\begin{equation}\label{b1}
Q_\phi(u)<0, \quad u=0\ \mathrm{ at }\ \partial \Sigma^*,\quad \int_{\Sigma^*} u\, dA_\phi=0.
\end{equation}
Notice that the weighted area element is 
$$dA_\phi=e^{\langle\Psi,e_3\rangle}\, d\Sigma=e^{\alpha_3(s)}\, d\Sigma=e^{\alpha_3(s)}\, ds\, dt.$$
By Proposition \ref{pr1}, there are no closed solutions of \eqref{1di}. We define next the compact piece $\Sigma^*$. Let $[a,b]$ be a subinterval of the domain of the base curve $\alpha$.   By the symmetry of $\Sigma$ in the direction of $\a$, we can suppose that the domain of $\Psi(s,t)$ in \eqref{para2} is $[a,b]\times [0,L]$ and let 
$\Sigma^*=\Psi([a,b]\times [0,L])$. The boundary of $\Sigma^*$ is given by
$$
\partial \Sigma^*=\left(\alpha([a,b])\times\{0,L\}\right)\cup \left(\{\alpha(a),\alpha(b)\}\times [0,L]\right).$$ 
A first attempt for the test function $u\in C_0^\infty(\Sigma^*)$  is to consider $u(s,t)=f(s)g(t)$ by separated variables.   Thus, the boundary condition $u=0$ in \eqref{b1} is equivalent to
$$
f(a)=f(b)=0,\quad g(0)=g(L)=0.
$$

 The condition for $u(s,t)$ to have zero weighted mean in \eqref{b1} reads as
\begin{equation}\label{bb2}
0= \int_a^b  f(s)e^{ \alpha_3(s)}\, ds \int_ 0^Lg(t)\, dt.
\end{equation}
This  condition is fulfilled if 
\begin{equation}\label{m-integral}
\int_a^b  f(s)e^{\alpha_3(s)}\, ds=0,\quad\mbox{or}\quad \int_ 0^Lg(t)\, dt=0.
\end{equation}
We now work on the terms that appear in \eqref{ll}. Since the first fundamental form of $\Sigma$ is $ds\, dt$, we have  
$$\nabla u=\langle\nabla u,\alpha'\rangle\alpha'+\langle\nabla u,e_2\rangle e_2=f'g\alpha'+fg'e_2.$$
Here, we understand that the derivative of each function is with respect to its variable. Thus,  $\langle\nabla u,\v\rangle =f'g \alpha_3'$. For the Laplacian of $u$, we have $\Delta u=f''g+fg''$. On the other hand, since a cylindrical surface is flat, the Gauss curvature $K$ of $\Sigma$ is zero. Therefore,  $ |A|^2 =H^2-2K=(\alpha_1'+\lambda)^2 $. Plugging these identities together in \eqref{ll}, the expression \eqref{sl} for $Q_\phi(u)$ is  
\begin{equation}\label{particular0}
Q_\phi(u)=-\int_a^b\int_0^Lfg\left(f''g+fg''+f'g\alpha_3'+H^2fg\right)e^{\alpha_3(s)}ds\, dt.
\end{equation}
The zero weighted mean condition \eqref{m-integral} will be obtained by means of the function $g$. For this, take 
\begin{equation}\label{gg}
g(t)=\sin\left(\frac{2\pi t}{L}\right),
\end{equation}
which vanishes at $t=0$ and $t=L$ and $g$ satisfies $\int_0^L g(t)\, dt=0$. Since $g''=-\frac{4\pi^2}{L^2}g$, by substituting into  \eqref{particular0} and integrating by parts, we have
\begin{equation}\label{particular}
\begin{split}
Q_\phi(u)&=-\int_0^Lg(t)^2\, dt \int_{a}^{b}f\left(f''+f'\alpha_3'+\left(H^2-\frac{4\pi^2}{L^2}\right)f\right)e^{\alpha_3}ds\\
&= \frac{L}{2}\int_{a}^{b}\left(f'^2-\left(H^2-\frac{4\pi^2}{L^2}\right)f^2\right)e^{\alpha_3}ds\\
&= \frac{L}{2}\int_{a}^{b}\left(f'^2-\left((\alpha_1'+\lambda)^2-\frac{4\pi^2}{L^2}\right)f^2\right)e^{\alpha_3}ds.
\end{split}
\end{equation}

At this point, we analyze the Plateau-Rayleigh instability criterion depending on the value of $\lambda$ and taking advantage of the explicit expression of the function $\alpha_1$. Recall that $\lambda\geq 0$. The case $\lambda=0$ corresponds to $\alpha$ being the grim reaper, hence it will be discarded since the surface is a graph on $\v^\top$ and thus it is strongly stable by Theorem \ref{t1}.

\subsection{Case $\lambda>1$}

  The explicit parametrization of the base curve $\alpha$ is given in \eqref{fp1}. By the periodicity of $\alpha$, we find 
$$\alpha(s+T)=\alpha(s)+(\alpha_1(T),0)=\alpha(s)-(\lambda s,0).$$
A fundamental piece of $\alpha$ is any subset of the type $\alpha([\bar{s},\bar{s}+T])$, for any $\bar{s}\in\R$. See Figure \ref{fig1}, left.

We will study the instability of $\Sigma$ only for fundamental pieces. Let  
$$s_0=\frac{T}{2}=\frac{\pi}{\sqrt{\lambda^2-1}}.$$
Since $\alpha$ is also symmetric with respect to the $x_3$-axis, then any fundamental piece of $\alpha$ is given by $\alpha([-s_0+\sigma,s_0+\sigma])$ for some $\sigma\in [0,s_0]$. Let 
\begin{equation*}
\Sigma(\sigma;L) = \{(\alpha_1(s),t,\alpha_3(s)):s\in[-s_0+\sigma,s_0+\sigma], t\in [0,L]\}.
\end{equation*}
Thus, each piece $\alpha([-s_0+\sigma,s_0+\sigma])$, with $\sigma\in [0,s_0]$ produces the entire curve $\alpha$ by repeating $\alpha([-s_0+\sigma,s_0+\sigma])$ along  the $x_1$-axis. Notice that the fundamental piece changes as $\sigma\in[0,s_0]$ varies. We assume that $\sigma$ is arbitrary but fixed and investigate the instability of the corresponding fundamental piece; see Figure \ref{fig1}.

 \begin{figure}[hbtp]
\centering
\includegraphics[width=.35\textwidth]{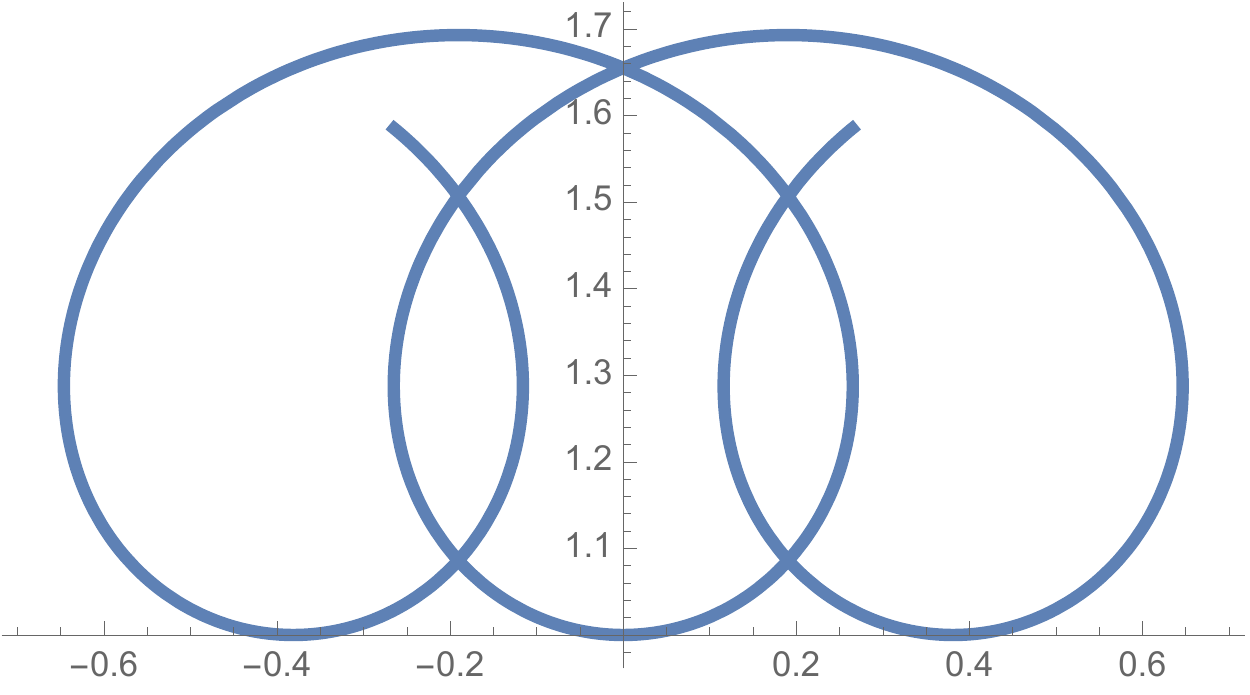}  \quad\includegraphics[width=.15\textwidth]{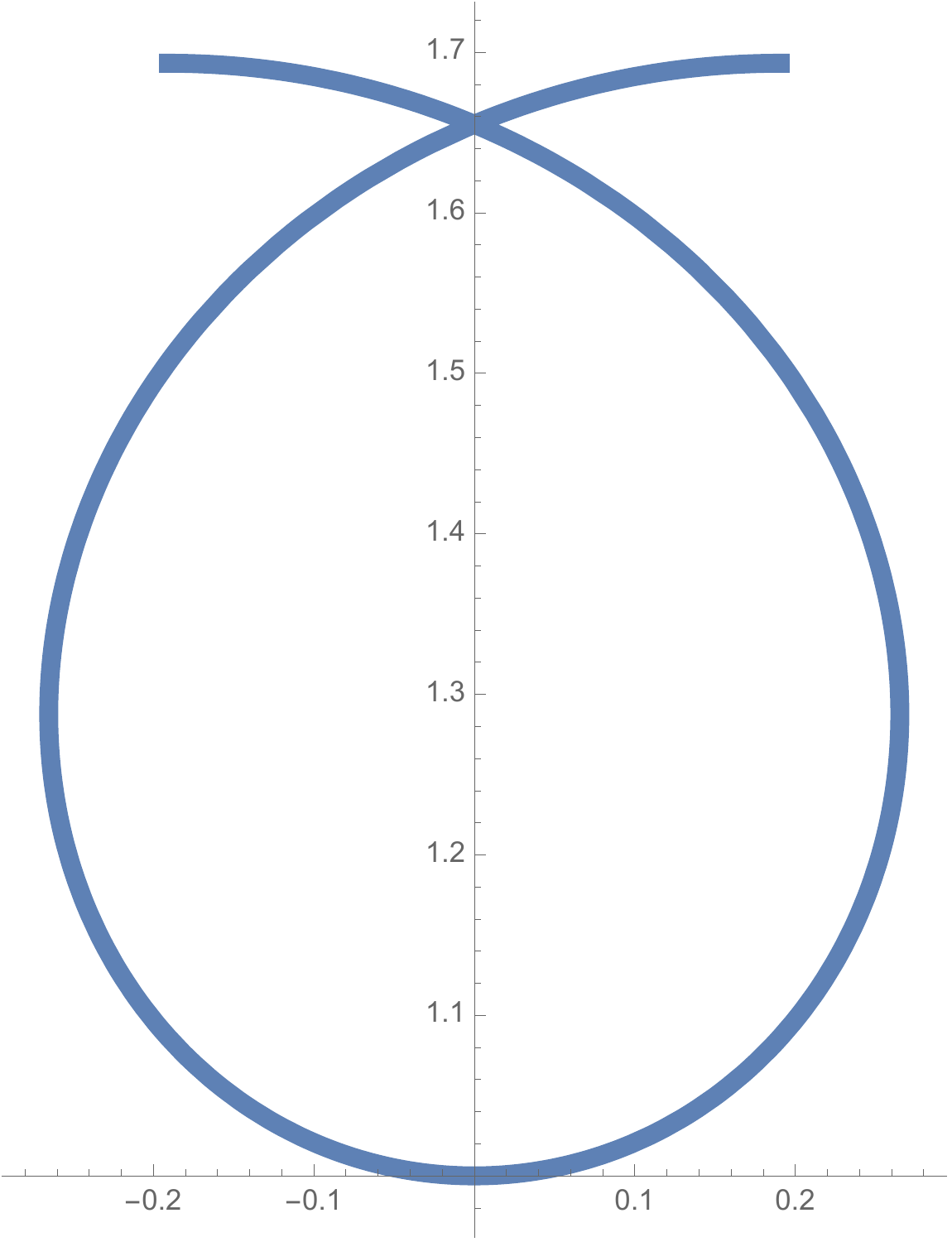}\quad\includegraphics[width=.25\textwidth]{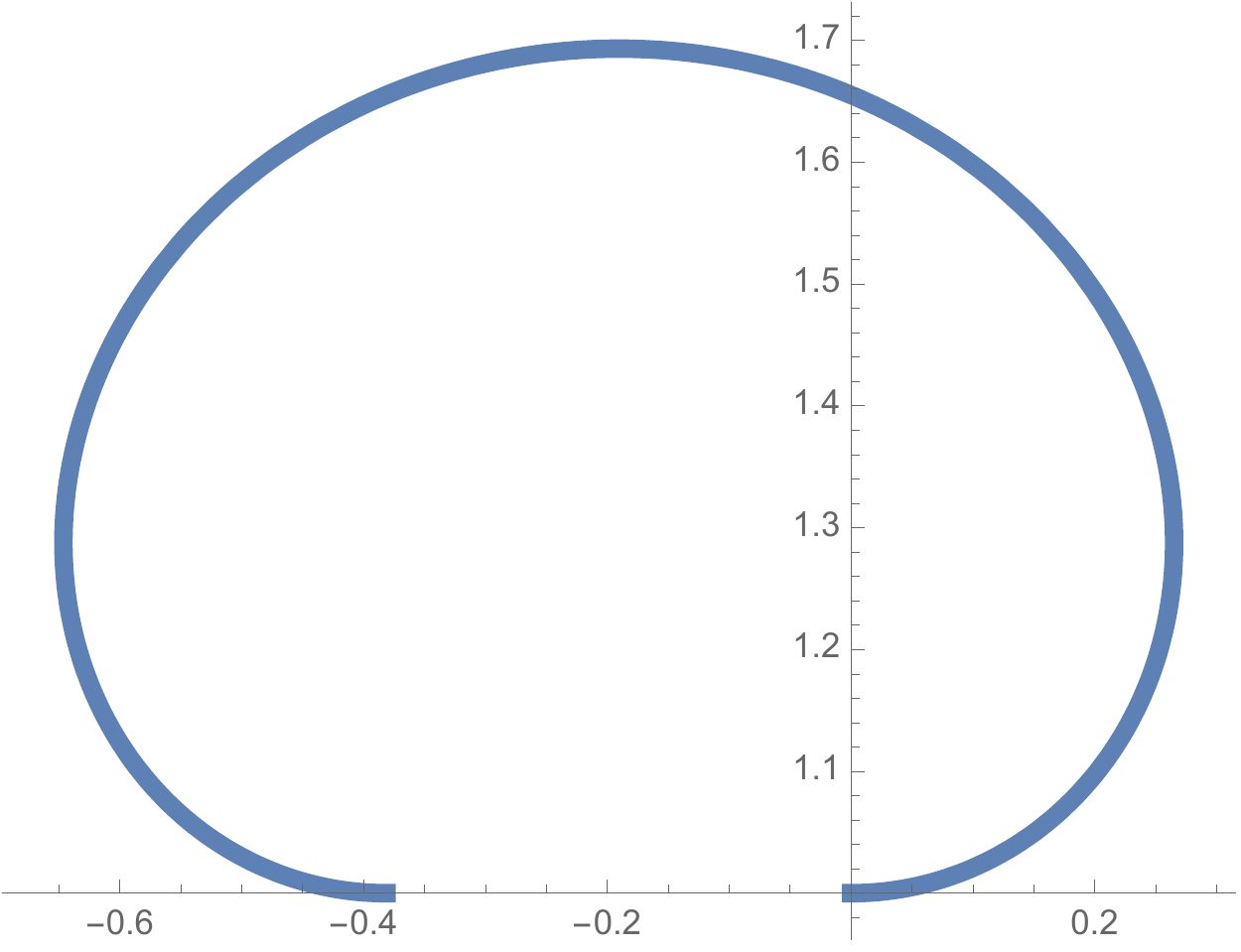} \caption{Case $\lambda>1$. Here $\lambda=3$. Different plots of $\alpha$: Full curve (left); $\alpha([-s_0,s_0])$ (middle); $\alpha([0,T])$ (right).} \label{fig1}
\end{figure}

For the function $f=f(s)$, we choose a function of type $\sin(s)$  which must vanish at $s=0$ and $s=\pi$. Notice that $f$ is defined in the interval $ [-s_0+\sigma,s_0+\sigma]$. Doing a suitable dilation and translation in the $s$ variable, we consider the   function
$$f(s)=\sin\left(\frac{\pi}{2s_0}s+\frac{\pi}{2}(1-\frac{ \sigma}{s_0})\right)e^{-\alpha_3(s)/2}.$$
The added term $e^{-\alpha_3/2}$ in the above expression of $f$ will help in the successive computations.  We split the quadratic form \eqref{particular} as
\begin{equation*}
Q_\phi(u)=\frac{L}{2}\int_{-s_0+\sigma}^{s_0+\sigma} f'^2e^{\alpha_3}ds-\frac{L}{2}\int_{-s_0+\sigma}^{s_0+\sigma}(\alpha_1'+\lambda)^2f^2e^{\alpha_3}dsds+\frac{2\pi^2}{L}\int_{-s_0+\sigma}^{s_0+\sigma}f^2e^{\alpha_3}ds.
\end{equation*}
After a straightforward computation,  
$$
\alpha_1'(s)+\lambda=\frac{\lambda^2-1}{\lambda-\cos(s\sqrt{\lambda^2-1})}.
$$
Substituting the function $f(s)$ into the quadratic form $Q_\phi(u)$ given in \eqref{particular}, explicit integration yields
\begin{eqnarray*}
\int_{-s_0+\sigma}^{s_0+\sigma} f'^2e^{\alpha_3}ds&=&\frac{\pi }{4}(\lambda+\cos(\sigma\sqrt{\lambda^2-1})),\\
\int_{-s_0+\sigma}^{s_0+\sigma}(\alpha_1'+\lambda)^2f^2e^{\alpha_3}ds&=&\pi (\lambda+\cos(\sigma\sqrt{\lambda^2-1})),\\
\int_{-s_0+\sigma}^{s_0+\sigma}f^2e^{\alpha_3}ds&=&\frac{\pi}{\sqrt{\lambda^2-1}}.
\end{eqnarray*}
Hence the value of $Q_\phi(u)$ is
\begin{equation}\label{q1}
Q_\phi(u)=\frac{\pi}{8L\sqrt{\lambda^2-1}}\left(16\pi^2-3L^2\sqrt{\lambda^2-1}(\lambda+\cos(\sigma\sqrt{\lambda^2-1}))\right).
\end{equation}
For values of $L$ close to $0$, $Q_\phi(u)$ is positive. However, viewed $Q_\phi(u)$ as a function on $L$, $Q_\phi(u)$ is monotically decreasing to $-\infty$. Equaling to 0 the above parenthesis, we have
\begin{equation}\label{l1}
L_0=\frac{4\pi}{\sqrt{3(\lambda+\cos(\sigma\sqrt{\lambda^2-1}))\sqrt{\lambda^2-1}}}.
\end{equation}
Therefore, we obtain the following instability result.

\begin{theorem}\label{thinestabilidad}
Suppose $\lambda>1$. Let be $L_0$ as in \eqref{l1} and $\sigma\in[0,s_0]$. If $L>L_0$, then the translating $\lambda$-soliton $\Sigma(\sigma;L)$ is unstable.
\end{theorem}

\begin{remark}
Let us view $L_0$ in \eqref{l1} as a function of $\lambda$. If $\lambda\rightarrow\infty$, then $L_0\rightarrow0$. On the contrary, if $\lambda\rightarrow1$, then $L_0\rightarrow\infty$. This is in accordance with the classical result for circular cylinders because for cylinders, the value of $L_0$ is $2\pi r$ and the constant mean curvature is $\lambda=1/r$.
\end{remark}

The values $L_0$ obtained in \eqref{q1} are strictly increasing for $\sigma\in[0,s_0]$, being the largest
\begin{equation}\label{L0uniforme}
L_0^{*}=\frac{4\pi}{\sqrt{3(\lambda-1)\sqrt{\lambda^2-1}}},
\end{equation}
for $\sigma=s_0$. As a matter of fact, we can give a uniform length, not depending on $\sigma$, for which every fundamental piece is unstable.

\begin{corollary}
Suppose $\lambda>1$. Let be $L_0^{*}$ as in \eqref{L0uniforme}. If $L>L_0^{*}$, then the translating $\lambda$-soliton $\Sigma(\sigma;L)$ is unstable, for every $\sigma\in[0,s_0]$.
\end{corollary}
 
\subsection{Case $\lambda=1$}

 The base curve  of cylindrical translating $1$-solitons  is given in \eqref{fp2}. This curve is symmetric about the $x_3$-axis and has one self-intersection point.  See Figure \ref{fig2}, left. The notation that we use is  
\begin{equation*} 
\Sigma(s_0;L)=\{(-s+2\arctan(s),t,\log(1+s^2)):s\in[-s_0,s_0], t\in [0,L]\} 
\end{equation*}
 The curve $\alpha$ is a graph in the interval $(-1,1)$ of the domain of $\alpha$ and its projection on the $x_1$-axis is $(1-\frac{\pi}{2},-1+\frac{\pi}{2})$. Thus, for any $s_0<1$, $\Sigma(s_0;L)$ is stable by Theorem \ref{t1}. Consequently, we will investigate the instability of symmetric compact pieces of $\alpha(s)$ for $s\in[-s_0,s_0],\ s_0>1$. 

\begin{theorem}\label{t42} Suppose $\lambda=1$.  Let $\bar{s}_0\sim 1.0213$. If $s_0>\bar{s}_0$, then there is 
\begin{equation}\label{l3}
L_0=8\pi s_0^{5/2}\left(15(\frac{s_0^3}{3}+9s_0-\left(9+6s_0^2-3 s_0^4\right) \tan ^{-1}(s_0))\right)^{-1/2}.
\end{equation}
  such that if $L>L_0$, the surface $\Sigma(s_0;L)$ is unstable.
\end{theorem}
\begin{proof}
 For the function $f=f(s)$ in the test function $u(s,t)=f(s)g(t)$, take 
$$f(s)=(s^2-s_0^2)e^{-\alpha_3(s)/2},$$
and $g$ is as in \eqref{gg}. The calculations of each one of the terms that appear in the expression \eqref{particular} of $Q_\phi$ are:

\begin{figure}[hbtp]
\centering
\includegraphics[width=.4\textwidth]{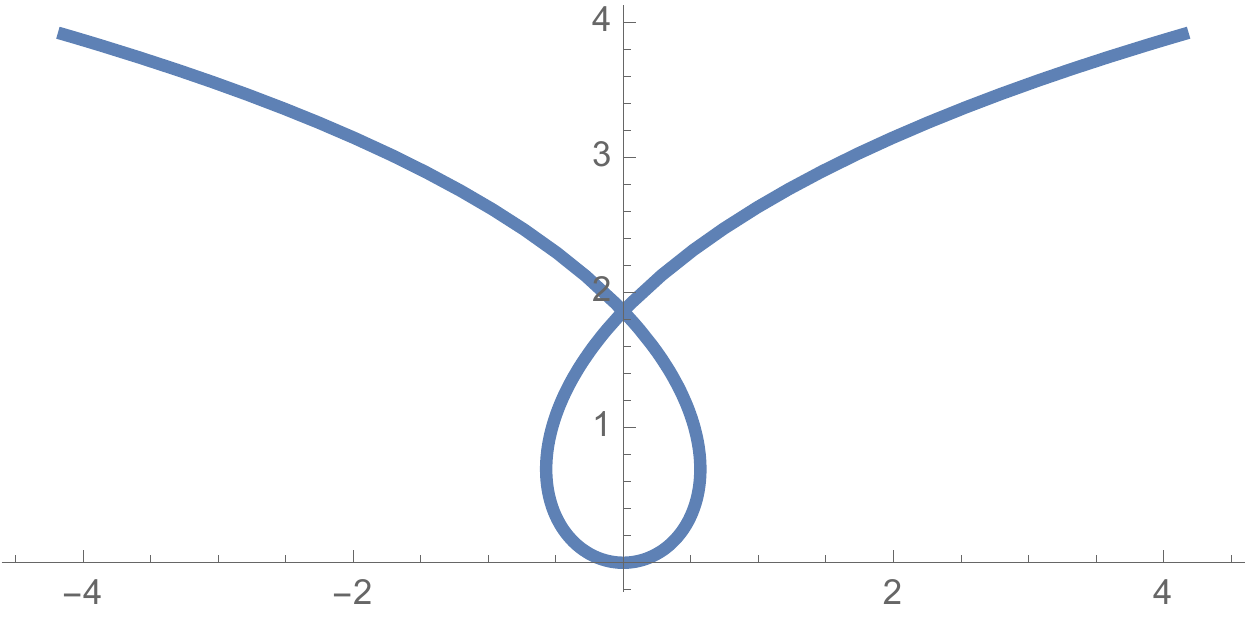} \qquad \includegraphics[width=.15\textwidth]{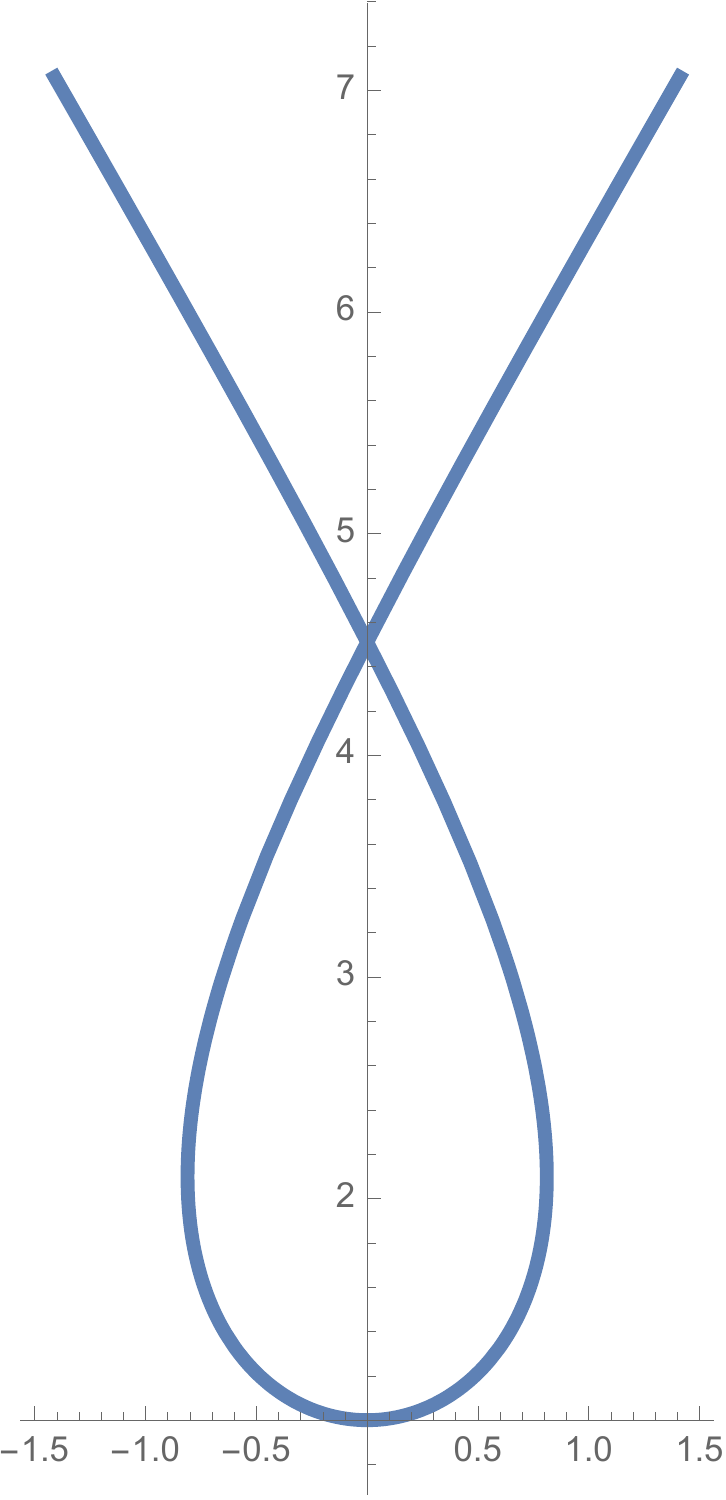} \caption{Case $\lambda=1$ (left) and $\lambda<1$ (right)} \label{fig2}
\end{figure}

\begin{align*}
\int_{-s_0}^{s_0} f'^2 e^{\alpha_3}\, ds&=\frac{11 s_0^3}{3}+3 s_0\left(s_0^4-2 s_0^2-3\right) \tan ^{-1}(s_0), \\
\int_{-s_0}^{s_0} (\alpha_1'+\lambda)^2f^2 e^{\alpha_3}\, ds&=-2 \left(s_0 \left(s_0^2+3\right)+\left(s_0^4-2 s_0^2-3\right) \tan ^{-1}(s_0)\right) \\
 \int_{-s_0}^{s_0} f^2 e^{\alpha_3}\, ds&= \frac{16  s_0^5}{15 }
\end{align*}
  Then
\begin{equation}\label{q2}
Q_\phi(u)=\frac{L}{2}\left( \frac{64 \pi ^2 s_0^5}{15 L^2}-\frac{s_0^3}{3}-9s_0+\left(9+6s_0^2-3 s_0^4\right) \tan ^{-1}(s_0)\right).
\end{equation}
  We study under what conditions the above parenthesis is negative. Let 
  $$h(L)=\frac{64 \pi ^2 s_0^5}{15 L^2}-\frac{s_0^3}{3}-9s_0+\left(9+6s_0^2-3 s_0^4\right) \tan ^{-1}(s_0).$$
  Notice that   $s_0$ is fixed. The function $h(L)$ is monotonically decreasing on $L$, with
  $$\lim_{L\to\infty}h(L)=-\frac{s_0^3}{3}-9s_0+\left(9+6s_0^2-3 s_0^4\right) \tan ^{-1}(s_0).$$ 
 We study when this limit is negative.  Define
  $$\varphi(s_0)=-\frac{s_0^3}{3}-9s_0+\left(9+6s_0^2-3 s_0^4\right) \tan ^{-1}(s_0).$$
  It is not difficult to see that this function $\varphi$ only vanishes at $\bar{s}_0\sim 1.0213$  and that for $s_0>\bar{s}_0$, the function $\varphi$ is negative. Thus, given $s_0>\bar{s}_0$, 
  $$\lim_{L\to\infty}h(L)=\varphi(s_0)<0.$$
Coming back to the expression of $Q_\phi(u)$ in \eqref{q2}, this implies that there is a unique $L_0>0$, which is given as the unique solution of the equation $h(L)=0$, such that if $L>L_0$, $Q_\phi(u)<0$, which proves the instability of the surface.
  \end{proof}
  
With a little more effort, if we see the value $L_0$ in \eqref{l3} as a function of $s_0$, it can be proved that $L_0$ is increasing on $s_0$. As a consequence, as the amplitude $s_0$ of $\alpha$ increases, longer and longer pieces of surfaces are required  to achieve the desired instability. 
\subsection{Case $\lambda<1$}
Suppose that $\lambda<1$.  From Proposition \ref{pr1} we study the instability of those surfaces that are not graphs on the $x_1$-axis. The curve  \eqref{fp3} is symmetric about the $x_3$-axis with a minimum at $s=0$ and with a unique self-intersection point. The curve $\alpha$ is a graph in $(-s_1,s_1)$, where 
\begin{equation}\label{ss1}
s_1=\frac{\cosh^{-1}(1/\lambda)}{\sqrt{1-\lambda^2}}.
\end{equation}

See Figure \ref{fig2}, right. We will study the stability of symmetric compact pieces of $\Sigma$, where $\alpha$ is defined in a symmetric interval $[-s_0,s_0]$. Again, let 
\begin{equation*} 
\Sigma(s_0;L)=\{(\alpha_1(s),t,\alpha_3(s)):s\in[-s_0,s_0], t\in [0,L]\}.
\end{equation*}
 We know that if   $s_0<s_1$, then $\Sigma(s_0;L)$ is stable by Theorem \ref{t1}. Therefore our interest is for those values of $s_0$ greater than $s_1$. 
 
For the test function, we take $g$ as in \eqref{gg} and 
$$f(s)=\cos\left(\frac{\pi s}{2s_0}\right)e^{-\alpha_3(s)/2}.$$
In contrast to the cases $\lambda>1$ and $\lambda=1$, an explicit integration of   \eqref{particular} is not possible.  For this reason, we will give a numerical study  for different values of $\lambda$ that shows that the Plateau-Rayleigh instability also holds for the case $\lambda<1$.  The numerical computations are performed using the software  {\tt Mathematica}.  As a sample of this study, we will consider the values   $\lambda=1/4$, $\lambda=1/2$ and $\lambda=3/4$.

The expression of $Q_\phi(u)$ is 
$$Q_\phi(u)=\frac{L}{2} I(u),$$
where 
$$I(u)=\int_{a}^{b}\left(f'^2-\left((\alpha_1'+\lambda)^2-\frac{4\pi^2}{L^2}\right)f^2\right)e^{\alpha_3}ds.$$
Once we have obtained the value $I(u)$,   we are looking for large values of $L$ such that $I(u)$ is negative. The routine with {\tt Mathematica} is the following. First, let us input the curve $\alpha$ given in \eqref{fp3}: 
 \begin{equation*}
\begin{split}
&\tt{  \alpha1[s\_]  := 2 \tan ^{-1}[\sqrt{\frac{\lambda +1}{1-\lambda }} \tanh[\frac{1}{2} \sqrt{1-\lambda ^2} s]]-\lambda  s;  }\\
&\tt{   \alpha3[s\_ ] := \log[\cosh [\sqrt{1-\lambda ^2} s]-\lambda]; }
\end{split}
\end{equation*} 
The test functions are 
  \begin{equation*}
\begin{split}
&\tt{ g[t\_]:=\sin [\frac{2 \pi  t}{L}];  }\\
& \tt{f[s\_]:=e^{-\frac{ \alpha3[s]}{2}} \cos[\frac{\pi  s}{2 \text{so}}];}
\end{split}
\end{equation*}

 The value $\lambda$ is fixed, for example, $\lambda=1/4$. The value $s_1$ is $s_1\sim2.1311$.  Next  we fix $s_0$ and obtain the value of $I(u)$ for different values of $L$. Here we use the   {\tt NIntegrate} command: 
\begin{equation*}
\begin{split}
&\tt{so=3;}\\
&\tt{\lambda =\frac{1}{4};}\\
&\tt{L=4;}\\
&\tt{  \text{NIntegrate}\left[e^{\alpha3[s]} \left(f'[s]^2+\frac{4 \pi ^2 f[s]^2}{L^2}+f[s]^2 \left(-\left(\lambda +\alpha1'[s]\right)^2\right)\right),\{s,-\text{so},\text{so}\}\right];}
\end{split}
\end{equation*}
In this case, we find
$$\tt{ Out[\%]= 7.1166}$$
In Tables \ref{table1}, \ref{table2} and \ref{table3},  we have computed the value of $I(u)$ for different values of $L$. As a first step, we need to know the value $s_1$ for which  $\alpha$ is not a graph 
 in a symmetric interval $[-s_0,s_0]$, with $s_0>1$. For the above choices of $\lambda$, the values of $s_1$ are  $2.1311$ ($\lambda=1/4$), $s_1=1.5206$ ($\lambda=1/2$) and   $ 1.2024$ ($\lambda=3/4$).  
 
 In the three tables, we have boxed the first value of $L$ such that $I(u)$ is negative. In the last choice of $s_0$, namely, $s_0=7$ ($\lambda=\frac14$), $s_0=10$ ($\lambda=\frac12$) and $s_0=10$ ($\lambda=\frac34$), the first value of $L$ for which $I(u)$ is negative does not appear in the table because it exceeds of the computed one: $L=40$ ($\lambda=\frac14$), $L=20$ ($\lambda=\frac12$) and $L=12$ ($\lambda=\frac34$).

\begin{table}[ht]
  \caption{Values of $I(u)$ for $\lambda=\frac14$.}\label{table1}

  \centering
  \begin{tabular}{ccccccc}
    \hline
 $L$   & $15$ &  $20$ & $25$ & $30$&$35$& $40$ \\ 
    \hline
    $s_0$ & & & & & &\\
 3 & 0.2405 & 0.0102 & \fbox{-0.0962} & -0.1541 & -0.1891 & -0.2117  \\ 
  4& 0.3229 & 0.0158 & \fbox{-0.1262} & -0.2034 & -0.2499 & -0.2802 \\
    5& 0.5434 & 0.1596 &\fbox{ -0.0180 }& -0.1145 & -0.1727 &  -0.2104\\
   6&  0.8320 & 0.3715 & 0.1583 & 0.0425 & \fbox{-0.0273 }& -0.0726\\
    7& 1.1585 & 0.6211 & 0.3724 & 0.2373 & 0.1558 & 0.1030  \\ 
  \hline
  \end{tabular}
  \end{table}

 \begin{table}[ht]
  \caption{Values of $I(u)$ for $\lambda=\frac12$.}\label{table2}

  \centering
  \begin{tabular}{ccccccc}
    \hline
 $L$   & $10$ &  $12$ & $14$ & $16$&$18$& $20$ \\ 
    \hline
    $s_0$ & & & & & &\\
 2 & 0.2486 & 0.0074 & \fbox{-0.1380} & -0.2324 & -0.2971 &  -0.3434  \\ 
  4& 0.3297 & \fbox{ -0.1527} & -0.4436 & -0.6324 & -0.7619 & -0.8545 \\
    6& 1.1578 & 0.4340 & \fbox{-0.0023} & -0.2856 & -0.4798 & -0.61872\\
   8&  2.1681 & 1.2030 & 0.6212 & 0.2435 & \fbox{-0.0153} & -0.2005\\
    10& 3.2460 & 2.0397 & 1.3124 & 0.8403 & 0.51667 & 0.2851  \\ 
  \hline
  \end{tabular}
  \end{table}

 \begin{table}[ht]
  \caption{Values of $I(u)$ for $\lambda=\frac34$.}\label{table3}

  \centering
  \begin{tabular}{ccccccc}
    \hline
 $L$   & $2$ &  $4$ & $6$ & $8$&$10$& $12$ \\ 
    \hline
    $s_0$ & & & & & &\\
 2 & 18.3781 & 3.5737& 0.832& \fbox{-0.1273}& -0.5715&-0.8127  \\ 
  4& 37.1420& 7.5332 & 2.0501 & 0.1310 & \fbox{-0.7572} & -1.2397 \\
    6& 56.7298 & 12.3166& 4.0918& 1.2132&\fbox{ -0.1191}& -0.8429\\
   8& 76.5192& 17.3016&6.3353& 2.4972& 0.7206&\fbox{ -0.2443}\\
    10& 96.3834& 22.3613& 8.6535& 3.8558& 1.6351& 0.4288  \\ 
  \hline
  \end{tabular}
  \end{table}

 The above results  give us convincing evidence of the existence of a critical value $L_0$ of the Plateau-Rayleigh instability criterion for the case $\lambda<1$. 
   
   \begin{theorem}[numerical] \label{t46}
Suppose $\lambda<1$. Let $s_1$ be the value given in \eqref{ss1}. Given $s_0>s_1$, then there is   $L_0>0$ depending on $s_0$ such that if $L>L_0$,  the translating $\lambda$-soliton $\Sigma(s_0;L)$ is unstable.
\end{theorem}

We finish with a remark about strong stability. The results obtained in this section refer to the notion of stability of translating $\lambda$-solitons, that is, assuming that the variations preserve the weighted volume. Similar results can be obtained if we investigate the problem of strong stability. In such a case, we are dropping the assumption $\int_\Sigma u\, dA_\phi=0$ in \eqref{b1}. For the choice of test functions by separation of variables,  this means that the conditions \eqref{m-integral} are not required. 

It is expectable that for strong stability the length of the surface reduces in comparison with the stable case.  Notice that for circular cylinders viewed as cmc surfaces, the critical value changes from $L_0=2\pi r$ to $L_1=\pi r$ for strong stability.

The computations are analogous with the only difference that now we choose $g(t)=\sin(\pi t/L)$, hence $g''=-\frac{\pi^2}{L^2}g$. This difference appears in the integral \eqref{particular0} for the expression $gg''$. Following the same steps, we have the following result.

\begin{corollary} 
\begin{enumerate}
\item Case $\lambda>1$. Let $L_0$ as in \eqref{l1}. Then, the translating $\lambda$-soliton $\Sigma (\sigma;L)$ is strongly unstable if $L>L_0/2$.
\item Case $\lambda=1$. Let be $\bar{s}_0\sim 1.0213$ and $L_0$ as in \eqref{l3}.  If $s_0>\bar{s}_0$, then the translating $1$-soliton  $ \Sigma(s_0;L)$ is strongly unstable if $L>L_0/2$.
\item Case $\lambda<1$. Let be $s_1$ and $L_0$ as in Theorem \ref{t46}.   If $s_0>s_1$, then the translating $\lambda$-soliton  $ \Sigma(s_0;L)$ is strongly unstable if $L>L_0/2$.
\end{enumerate}

\end{corollary}
 
\section{The stability of circular cylinders}\label{sec5}

In this section, we investigate the Plateau-Rayleigh instability phenomenon for circular cylinders. We have seen in Example \ref{r2} that circular cylinders are the only non-planar surfaces that are both cmc surfaces and translating $\lambda$-solitons. For a circular cylinder,  the density vector $\v$ coincides with the direction of the rulings. Here, $\lambda=1/r$, where $r>0$ is the radius of the cylinder. Therefore, circular cylinders appear as a case of special interest. First, because we can compare the same surface with two different notions of stability, one as cmc surface and other as a translating $\lambda$-soliton. Furthermore, in the latter case, the surface does not appear in Proposition \ref{pr1} because the density vector $\v$ is not orthogonal to the rulings.  A third interesting aspect is that the base curve, a circle, is closed, which it does not appeared in the discussion of Section \ref{sec4}. 

In this section, we will address the instability of compact pieces of a circular cylinder where the base curve is a full circle. Without loss of generality, we assume $\v=\a=(0,0,1)$. Let us parametrize the base curve  as  $\alpha(s)=r(\cos(s/r), \sin(s/r),0)$. Then, the circular cylinder $C_r$, whose base curve is $\alpha$, is parametrized by 
 $\Psi(s,t)=\alpha(s)+t\, e_3$. Let $C_r(L)$  denote the compact piece of $C_r$ of length $L$, 
 $$C_r(L)=\Psi(\R\times[0,L]).$$
 Consider a test function $u$ of type  $u(s,t)=f(s)g(t)$. Since now $\alpha$ is a closed curve, the function $f(s)$  must obey the boundary conditions 
 \begin{equation}\label{f3}
 f(0)=f(2\pi r)\quad\mbox{and}\quad f'(0)=f'(2\pi r).
 \end{equation}

Because we are interested in comparing the different notions of stability for the same surface, we briefly recall   the Plateau-Rayleigh instability criterion of a circular cylinder as a cmc surface.  In the case of constant mean curvature, the argument is as in Section \ref{sec4}   but now the quadratic form is 
$$
Q(u)=-\int_0^{2\pi r}  \int_0^Lfg\left(f''g+fg''+\frac{1}{r^2}fg\right)\, dsdt.
$$
 Taking $f(s)=1$ and with the same function $g(t)=\sin(2\pi t/L)$,   we obtain 
$$Q(u)=-\pi r L\left(\frac{1}{r^2}-\frac{4\pi^2}{L^2}\right).$$
If we take $L_0=2\pi r$, then $C_r(L)$ is unstable for all $L>L_0$.

We now view   a circular cylinder $C_r$ as a translating $\lambda$-soliton. The quadratic form $Q_\phi$ is different from the expression \eqref{particular} because  the weighted area element is 
$$dA_\phi=e^{\langle\Psi,e_3\rangle}\, d\Sigma=e^t\, ds\, dt$$
and $\langle\nabla u,\v\rangle=fg'$.  Then, integrating by parts we have
\begin{equation*}
\begin{split}
Q_\phi(u)&=-\int_0^{2\pi r} \int_0^L fg\left(f''g+fg''+fg'+\frac{1}{r^2} fg\right)e^t\, ds\, dt\\
&=\int_0^{2\pi r}f'^2\, ds\int_0^L g^2e^t\, dt-\int_0^{2\pi r}f^2\, ds\int_0^Lg(g''+g'+\frac{1}{r^2}g)e^t\, dt.
\end{split}
\end{equation*}
Now, the boundary conditions are \eqref{f3} but the zero weighted mean condition \eqref{bb2}  must be modified to
\begin{equation}\label{bb3}
0= \int_a^b  f(s)\, ds \int_ 0^Lg(t)e^{ t}\, dt.
\end{equation}
This time, we choose as test functions $f(s)=1$ and
\begin{equation}\label{g3}
g(t)=\sin\left(\frac{2\pi t}{L}\right)e^{-t}.
\end{equation}

 \begin{theorem}\label{tc} Let $r<\sqrt{2}$. Consider a circular cylinder of radius $r>0$ and define
 $$L_0=\frac{\sqrt{8}\pi r}{\sqrt{2-r^2}}.$$
 Then, $C_r(L)$ is unstable as a translating $\frac{1}{r}$-soliton if $L>L_0$.
 \end{theorem}
 
 \begin{proof} With the choice of the above  test function, explicit integration of  $Q_\phi(u)$ yields
$$
Q_\phi(u)=\frac{8\pi^3e^{-L}(1-e^{-L})}{rL^2(L^2+16\pi^2)}\left(8\pi^2r^2+L^2(r^2-2)\right).
$$
Then, the result is immediate.
\end{proof}

We finish this section with further discussions on the instability of circular cylinders and the value $r=\sqrt{2}$ that appears in Theorem \ref{tc}. In \cite[Th. 4.8]{lo1}, the second author classified all the rotational translating $\lambda$-solitons by means of the qualitative study of the solutions of a nonlinear autonomous system. In particular, vertical circular cylinders were proved to be exactly the equilibrium points of this differential system. Furthermore,  
\begin{enumerate}
\item If $\lambda>\frac12$, then any rotational translating $\lambda$-soliton close enough to the cylinder converges asymptotically to it,  wiggling  around infinitely-many times and in particular never being a graph outside a compact set.
\item If $\lambda\leq\frac12$, then any rotational translating $\lambda$-soliton close enough to the cylinder is a graph outside a compact set and converges asymptotically to it.
\end{enumerate}
Therefore, if $\lambda\leq\frac12$, any infinitesimal smooth deformation of the cylinder is close enough to a graphical translating $\lambda$-soliton, hence stable. On the contrary, if $\lambda>\frac12$ then the infinitesimal deformation of the cylinder is close enough to a non-graphical translating $\lambda$-soliton, hence the desired instability may occur. In Theorem \ref{tc},  we achieved the  instability for all values $\lambda>1/\sqrt{2}$, which are in particular greater than $\frac12$. Bearing these discussions in mind, it is natural to think that the Plateau-Rayleigh instability phenomenon holds when $\lambda>1/2$.  

The authors have also considered another different test functions. For example, instead $f(s)=1$ and $g(t)$ as in \eqref{g3}, we have the  following values of $Q(u)$ for the  next two test functions:
\begin{align*}
&f(s)=\sin\left(\frac{s}{r}\right), & &g(t)= \sin\left(\frac{\pi t}{L}\right),& &Q(u)=\frac{r\pi^3(e^{L}-1)(L^2+2\pi^2)}{L^4+4\pi^2 L^2}.\\
&f(s)=\sin\left(\frac{s}{r}\right), & &g(t)= \sin\left(\frac{\pi t}{L}\right)e^{-t},& &Q(u)=\frac{r\pi^3 (1-e^{-L})(L^2+2\pi^2)}{L^4+4\pi^2 L^2}.
\end{align*}
Now, condition \eqref{bb3} holds thanks to the function $f(s)$. Notice that   the value of $Q_\phi(u)$ is positive in both cases {\it for any $L$}.

\section*{Acknowledgements}  

Antonio Bueno has been partially supported by the Project P18-FR-4049. 

Rafael L\'opez is a member of the Institute of Mathematics  of the University of Granada and he   has been partially supported by  the Projects  PID2020-117868GB-I00 and MCIN/AEI/10.13039/501100011033.


\def\refname{References}

\end{document}